\documentclass[12pt,a4paper]{amsart}
 \usepackage{amssymb,amsthm}
\usepackage{amssymb}
\usepackage{amsthm}
\usepackage{latexsym}
\usepackage{pstricks}
\usepackage{amsbsy}
\usepackage{tikz-cd}

\usepackage{array}
\usepackage[all]{xypic}
\usepackage{amscd}
\usepackage{mathrsfs}
\usepackage{amsfonts}
\usepackage{graphicx}
\usepackage{url}
\usepackage{bookmark}
\usepackage{mathtools}

\usepackage{tikz}
\usepackage{enumerate,multicol}
\usetikzlibrary{arrows}
\usetikzlibrary{positioning}
\xyoption{matrix}
\xyoption{arrow}

\usepackage{amsthm}
\usepackage{mathtools}
\usepackage{graphicx}
\usepackage{amssymb,amsthm}
\usepackage{latexsym}
\usepackage{pstricks}
\usepackage{amsbsy}
\usepackage{tikz-cd}
\usepackage{array}
\usepackage[all]{xypic}
\usepackage{amscd}
\usepackage{mathrsfs}
\usepackage{amsfonts}
\usepackage{url}
\usepackage{bookmark}
\usepackage{mathtools}
\usepackage{caption}
\usepackage{float}
\usepackage{tikz}
\usetikzlibrary{shapes}

\usepackage{enumerate,multicol}
\usetikzlibrary{arrows}
\usetikzlibrary{positioning}
\xyoption{matrix}
\xyoption{arrow}

\newcommand{\RomanNumeralCaps}[1]{\MakeUppercase{\romannumeral #1}}

\newtheorem{theorem}{Theorem}[section]
\newtheorem*{theorem*}{Theorem}
\newtheorem{lemma}[theorem]{Lemma}

\newtheorem{proposition}[theorem]{Proposition}
\newtheorem{definition}[theorem]{Definition}
\theoremstyle{definition}
\newtheorem{remark}[theorem]{Remark}

\begin{document}
 
 \providecommand{\Gen}{\mathop{\rm Gen}\nolimits}%
\providecommand{\cone}{\mathop{\rm cone}\nolimits}%
\providecommand{\Sub}{\mathop{\rm Sub}\nolimits}%
\providecommand{\rad}{\mathop{\rm rad}\nolimits}%
\providecommand{\coker}{\mathop{\rm coker}\nolimits}%
\def\A{\mathcal{A}}
\def\C{\mathcal{C}}
\def\D{\mathcal{D}}
\def\P{\mathcal{P}}
\def\X{\mathbb{X}}
\def\Y{\mathbb{Y}}
\def\E{\mathcal{E}}
\def\Z{\mathcal{Z}}
\def\K{\mathcal{K}}
\def\F{\mathcal{F}}
\def\T{\mathcal{T}}
\def\U{\mathcal{U}}
\def\V{\mathcal{V}}
\def\W{\mathsf{W}}
\def\PP{{\mathbb P}}
\providecommand{\add}{\mathop{\rm add}\nolimits}%
\providecommand{\End}{\mathop{\rm End}\nolimits}%
\providecommand{\Ext}{\mathop{\rm Ext}\nolimits}%
\providecommand{\Hom}{\mathop{\rm Hom}\nolimits}%
\providecommand{\ind}{\mathop{\rm ind}\nolimits}%
\newcommand{\module}{\mathop{\rm mod}\nolimits}%

\newcommand{\sbt}{\,\begin{picture}(-1,1)(-1,-3)\circle*{3}\end{picture}\ }

\title[Presentations of Groups with Even Length Relations]{Presentations of Groups with Even Length Relations}

\author{Dr Isobel Webster}
\address{
              University of St Andrews \\
              Tel.: +447720987969\\
             % Fax: +123-45-678910\\
             \\
}
\email{iw49@st-andrews.ac.uk}
\date{July 2021}
 \maketitle
\begin{abstract}
We study the properties of groups that have presentations in which the generating set is a fixed set of involutions and all additional relations are of even length. We consider the parabolic subgroups of such a group and show that every element has a factorisation with respect to a given parabolic subgroup. Furthermore, we give a counterexample, using a cluster group presentation, which demonstrates that this factorisation is not necessarily unique.  
\end{abstract}
\thanks{ }

\maketitle
\section{introduction}\label{intro}

A presentation of a group is a concise method of defining a group in terms of generators and relations. In special cases, much information about the corresponding group can be extracted from a given presentation \cite{MKS}. Coxeter presentations are a classical example of this \cite{Bjorner}. 

Recall that a pair $(W,S)$, where $S =  \lbrace s_{1},.., s_{n} \rbrace$ is a non-empty finite set and $W$ is a group, is called a Coxeter system if $W$ has a group presentation with generating set $S$ subject to relations of the form $(s_{i}s_{j})^{m(i, j)}$, for all $s_{i}, s_{j} \in S$, with $m(i, j) = 1$ if $i=j$ and $m(i, j) \geq 2$ otherwise, where no relation occurs on $s_{i}$ and $s_{j}$ if $m(i, j) = \infty$ \cite[Section 5.1]{Humphreys}. Such a group, $W$, is called a Coxeter group. 

An outline of the rich history of research into these groups is given in \cite[Historical Note]{Bourbaki}. In particular, it is well known that the finite Coxeter groups can be classified via their Coxeter graphs and the class of finite Coxeter groups is precisely the class of finite reflection groups \cite[Chapter VI, Section 4, Theorem 1]{Bourbaki}, \cite{Coxeter2}, \cite{Coxeter}. The applications of Coxeter groups are widespread throughout algebra \cite{Bourbaki}, analysis \cite{Wavelets}, applied mathematics \cite{QM} and geometry \cite{FT2}. However, the many combinatorial properties of Coxeter groups make them an interesting topic of research in their own right (see \cite{Bjorner}).

%For example, \cite{Coxeter2} showed that each finite reflection group can be defined by a Coxeter presentation and \cite{Coxeter} later showed that, conversely, every finite Coxeter group has a realisation as a reflection group. Any Coxeter system can be represented by its Coxeter graph, which defines the Coxeter presentation of the group. The Coxeter graphs of the irreducible finite Coxeter systems are well-known and these graphs classify the finite Coxeter groups \cite[Chapter VI, Section 4, Theorem 1]{Bourbaki}. 

For example, for any subset $I \subseteq S$, $W_{I}$ denotes the subgroup of $W$ generated by $I$. Any subgroup of $W$ which can be obtained in this way is called a parabolic subgroup of $W$ \cite[Section 5.4]{Humphreys}. It is known that the coset $wW_{I}$ contains a unique element of minimal length for each $w \in W$, meaning that we can choose a distinguished coset representative of $wW_{I}$. It follows that there exists a unique factorisation of each element with respect to a given parabolic subgroup \cite[Proposition 2.4.4]{Bjorner}. In certain cases, the reduced expressions for these distinguished coset representatives for a maximal parabolic subgroup can be described explicitly; see for example \cite[Corollary 3.3]{RMThesis}, \cite{Papi}.

%The reduced expressions for these distinguished coset representatives for a maximal parabolic subgroup of a Coxeter group of type $A_{n}$ are well known and can be described explicitly; see for example \cite[Corollary 3.3]{RMThesis}, \cite{Papi}.

%Reduced expressions of elements of a Coxeter group have been studied more generally yielding results such as the Exchange Lemma \cite[Section 1.5]{Bjorner}. 

%In particular, the Exchange Lemma is a fundamental property of a Coxeter group  which states that a pair $(W, S)$ is a Coxeter system if and only if each element of $S$ is of order 2 and $(W, S)$ satisfies the Exchange Property \cite[Section 1.5]{Bjorner}.

In this paper, we will consider group presentations which generalise the Coxeter case by allowing generating sets of infinite size and any relations that have even length. We show that a variation of this property of Coxeter groups holds for any group, $G$, which has a group presentation of this type. 

In particular, we can define a parabolic subgroup of $G$ in an analogous way to the Coxeter case and prove that there exists a (not necessarily unique) factorisation of each element of $G$ with respect to a given parabolic subgroup. We also give a counterexample using cluster group presentations (in the sense of \cite[Definition 1.2]{Webster}) showing that, in contrast to the Coxeter case, this factorisation is not unique in general.  

%Given a quiver appearing in a skew-symmetric cluster algebra, also known as a cluster quiver, the corresponding cluster group presentation \cite[Definition 1.2]{Webster} (see also \cite{BM} and \cite{GM}) will have relations of even length. Thus the results outlined above will hold for cluster groups. We give an example to show that the factorisation mentioned above is not unique in general for cluster groups, in contrast with the Coxeter case. 

Some of these results form part of the author’s Ph.D. thesis \cite{Webster2}, carried out at the University of Leeds. The paper will proceed in the following way. In Section $\ref{sec1}$ we establish basic properties of the length function on $G$ and provide a more detailed summary of the main result of this paper. In Section $\ref{properties}$ we prove our first main result: that there exists a factorisation of each element of $G$ with respect to a given parabolic subgroup. In Section $\ref{clustergroups}$, we construct an example in which the factorisation with respect to a given parabolic subgroup is not unique. 

\section{Group presentations with even length relations.}\label{sec1}
Let $G$ be a group arising from a group presentation $\langle X | R \rangle$, where $X$ may be finite or infinite. By the definition of the presentation of a group, any element $w$ of $G$ can be written as $$w = x_{1}^{a_{1}}x_{2}^{a_{2}}...x_{r}^{a_{r}},$$ where $r \in \mathbb{N}$, $x_{j} \in X$ and $a_{j} = \pm 1$, for all $1 \leq j \leq r$. The \textbf{length} of $w \in G$, $l(w)$, is the smallest $r$ such that $w$ has an expression of this form and a \textbf{reduced expression} of $w$ is any expression of $w$ as a product of $l(w)$ elements of $X\cup X^{-1}$ (where  $X^{-1} = \lbrace x^{-1} : x \in X \rbrace$ is a copy of $X$). We refer to $l$ as the \textbf{length function} on $G$. Taking an index set $I$, if $R$ is a set of relations of the form $u_{i} = v_{i}$ for $i \in I$, with $u_{i}, v_{i} \in F(X)$, where $F(X)$ denotes the free group on $X$, then the \textbf{length of the relation $u_{i} = v_{i}$} is given by the length of the word $u_{i}v_{i}^{-1}$ in $F(X)$.

Analogously to the Coxeter case, we define a parabolic subgroup of $G$. 

\begin{definition} 

For $I \subseteq X$, we denote by $G_{I}$ the subgroup of G generated by $I$. A (standard) \textbf{parabolic subgroup} of $G$ is a subgroup of the form $G_{I}$ for some $I \subseteq X$.
\end{definition}

Moreover, for each $I \subseteq X$ we define the following sets. $$G^{I} = \lbrace w \in G : l(wx) > l(w) \hspace{0.15cm} \forall \hspace{0.15cm} x \in I \rbrace;$$
$$^{I}	G = \lbrace w \in G : l(xw) > l(w) \hspace{0.15cm} \forall \hspace{0.15cm} x \in I \rbrace.$$

In Section \ref{properties}, we will prove our main result: 

\begin{proposition}\label{decomposition} Let $G$ be a group generated by a fixed set, $X$, of involutions subject only to relations of even length. For $I \subseteq X$, let $G_{I}$ denote the subgroup of $G$ generated by $I$ and $	G^{I} = \lbrace w \in G : l(wx) > l(w) \hspace{0.25cm} \forall \hspace{0.25cm} x \in I \rbrace$. Then every element $w \in G$ has a factorisation $$w = ab,$$ where $a \in G^{I}, b \in G_{I}$ and $l(w) = l(a) + l(b)$.
\end{proposition} This result is a strengthening of \cite[Proposition 7.2.4]{Webster2}. By comparison, we can see that this result is similar to \cite[Proposition 2.4.4]{Bjorner} for Coxeter groups. However, unlike the Coxeter case, we will show in Section $\ref{clustergroups}$ that the factorisation for elements of $G$ with respect to a given parabolic subgroup is not necessarily unique or determined by minimal length elements of the coset $wG_{I}$. 

%using a cluster group example...

%Coxeter presentations are a special case of group presentations of this type (i.e. group presentations satisfying conditions $(\RomanNumeralCaps{1})$ and $(\RomanNumeralCaps{2})$). Similarly, given a quiver appearing in a skew-symmetric cluster algebra, also known as a cluster quiver, the corresponding cluster group presentation \cite[Definition 1.2]{Webster} satisfies conditions $(\RomanNumeralCaps{1})$ and $(\RomanNumeralCaps{2})$, meaning the results proven in this paper demonstrate that a cluster group presentation arising from any cluster quiver possesses properties which are comparable to those of Coxeter presentations. 

\section{Proof of main result.}\label{properties}

To begin, we prove some basic results for the length function on a group with an arbitrary group presentation $\langle X | R \rangle$.  
\begin{lemma}\label{Exchange Lemma 2} If $x \in X$ and $w \in G$ then $l(xw) < l(w)$ if and only if there exists a reduced expression of $w$ beginning in $x^{-1}$. 
\end{lemma}
\begin{proof} Let $l(xw) < l(w)$. Suppose $xw = x_{1}^{a_{1}}x_{2}^{a_{2}}...x_{r}^{a_{r}}$ is a reduced expression, where $x_{j} \in X$ with $a_{j} = \pm 1$ for all $1 \leq j \leq r$. Then $w = x^{-1}x_{1}^{a_{1}}x_{2}^{a_{2}}...x_{r}^{a_{r}}$ is an expression of $w$ of length $r+1$. Moreover, this expression must be reduced otherwise $l(w) \leq r$, contradicting that $l(xw) < l(w)$.

%If $l(x_{i}w) < l(w)$ then $l(w^{-1}x_{i}) < l(w^{-1})$. Let $I = \lbrace x_{i} \rbrace$ and consider the parabolic subgroup $G_{I} = \lbrace e, x_{i} \rbrace$. By Proposition $\ref{decomposition}$, there exists a factorisation $w^{-1}=ab$ where $a \in G^{I}, b \in G_{I}$ and $l(a)+l(b)=l(w)$.  As $l(w^{-1}x_{i}) < l(w^{-1})$, we must have that $w^{-1} \notin G^{I}$, meaning $b\neq e$, thus $b =x_{i}$. That is, $w^{-1}=ax_{i}$ where $a \in G^{I}$ and $l(w^{-1}) = l(b) + 1$. Taking a reduced expression $a=x'_{i_{1}}...x'_{i_{k}}$, we obtain a reduced expression $w^{-1}=x'_{i_{1}}...x'_{i_{k}}x_{i}$, meaning $w=x_{i}x'_{i_{k}}...x'_{i_{1}}$ is a reduced expression of $w$ beginning with $x_{i}$.  

Conversely, if there exists a reduced expression of $w$ beginning in $x^{-1}$, say $w = x^{-1}x_{2}^{a_{2}}...x_{r}^{a_{r}}$ where $x_{j} \in X$ with $a_{j} = \pm 1$ for all $2 \leq j \leq r$, then $xw = x_{2}^{a_{2}}...x_{r}^{a_{r}}$ and so $l(xw) \leq r-1 < l(w)$. 
\end{proof}

For the remaining results, we let $G$ be a group with group presentation $\langle X | R \rangle$ which satisfies the following conditions. 

\begin{itemize}
    \item[(\RomanNumeralCaps{1})] $X$ is a fixed set of involutions. That is, for each $x \in X$, $x^{2} = e$. 
    \item[(\RomanNumeralCaps{2})] Every relation in $R$ has even length. 
\end{itemize}

We establish some properties, analogous to the Coxeter case, of the length function on $G$. The first result is a consequence of Lemma $\ref{Exchange Lemma 2}$.

\begin{lemma}\label{IG} Let $I = X \setminus \lbrace x \rbrace$ for some $x \in X$ and take $w \in G$ such that $w \neq e$. Then $w \in$ $^I G $ if and only if all reduced expressions of $w$ begin in $x$.  
\end{lemma}
\begin{proof} If $w \in$ $ ^I G$ has a reduced expression beginning in $y$ for some $y \in X$ such that $y \neq x$ then $l(yw) < l(w)$, contradicting that $w \in$ $ ^I G$. Conversely, suppose $w \in G$ is such that all reduced expressions of $w$ begin in $x$. For any $y \in X$ such that $l(yw) < l(w)$ there exists a reduced expression of $w$ beginning in $y$, by Lemma $\ref{Exchange Lemma 2}$. Hence $y=x$ and so $w \in$  $^I G$. 

%We conclude that any element of $G$ not equal to the identity element lies in $ ^{i}	G$ if and only if all of its reduced expressions begin in $x_{i}$. 

%Take any $w \in G$ and suppose $x_{i} \in X$ is such that $l(x_{i}w) < l(w)$. 
%By Lemma $\ref{Exchange Lemma 2}$ we may assume that there exists a reduced expression $w = x_{i_{1}}...x_{i_{r}}$ such that $x_{i}x_{i_{1}}...x_{i_{j-1}} = x_{i_{1}}...x_{i_{j-1}}x_{i_{j}}$, for some $1 \leq j \leq r$, by taking a reduced expression $w = x_{i_{1}}...x_{i_{r}}$ such that $x_{i_{1}} = x_{i}$ and choosing $j=1$. That is, $x_{i} = x_{i_{1}}...x_{i_{j-1}}x_{i_{j}}x_{i_{j-1}}...x_{i_{1}}$. Thus \begin{align*}
%x_{i}w &= (x_{i_{1}}...x_{i_{j-1}}x_{i_{j}}x_{i_{j-1}}...x_{i_{1}})x_{i_{1}}...x_{i_{r}} \\
%&= x_{i_{1}}...x_{i_{j-1}}x_{i_{j+1}}...x_{i_{r}}\\
%&=x_{i_{1}}...\hat{x}_{i_{j}}...x_{i_{r}}
%\end{align*}
\end{proof} 

\begin{remark} It follows from Lemma \ref{IG} that $w \in$ $G^{I} $ if and only if all reduced expressions of $w$ end in $x$.

\end{remark}

The following result is analogous to \cite[Proposition  5.1]{Humphreys} for Coxeter groups.

\begin{proposition}\label{5.1} There exists a surjective homomorphism $\varepsilon: G \longrightarrow \lbrace \pm 1 \rbrace$ defined by $\varepsilon: x \longmapsto -1$ for each $x \in X$. It follows that the order of each generator is 2.
\end{proposition}
%begin{proof} We define a map from $F(X)$ and show this induces a group homomorphism on $G$. We begin by defining a group homomorphism:\begin{align*}
  %  \tilde{\epsilon}: F(X) &\longrightarrow \lbrace \pm 1 \rbrace \\
 %   \tilde{\epsilon}: x_{i} &\longmapsto -1.
%\end{align*} where $\lbrace \pm 1 \rbrace$ is considered as a multiplicative group. To show that $\tilde{\epsilon}$ induces $\epsilon$, we must show that each relation in $R$ lies in the kernel of $\tilde{\epsilon}$. 

%By assumption, each relation $r \in R$ is of even length in $F(X)$. This means that $r = x_{i_{1}}x_{i_{2}}...x_{i_{k}},$ is a reduced word where $x_{i_{j}} \in X\sqcup X^{-1}$ for all $1 \leq j \leq k$ (so $x_{i_{j}}$ and $x_{i_{j}}^{-1}$ are never adjacent for any $x_{i_{j}} \in X$) and $k = 2n$ for some $n \in \mathbb{N}$. Thus $\tilde{\epsilon}(r) = \tilde{\epsilon}(x_{i_{1}})\tilde{\epsilon}(x_{i_{2}})...\tilde{\epsilon}(x_{i_{k}}) = (-1)^{k} = (-1)^{2n} = 1^{n} = 1 $.

%Thus $\epsilon$ is a surjective group homomorphism. As -1 has order 2 in $\lbrace \pm 1 \rbrace$, each generator, $x_{i}$, of $G$ is of order 2. 
% \end{proof}

\begin{remark} Suppose the group presentation $\langle X | R \rangle$ satisfies only condition $(\RomanNumeralCaps{2})$ and each generator has finite order. In this case, the surjective homomorphism $\varepsilon$ exists and the order of each generator $x \in X$ of $G$ is even.
\end{remark}

In \cite[Section 5.2]{Humphreys}, the length function on a Coxeter group is defined along with five basic properties. Below, we consider the length function on $G$.

As for the Coxeter case, each element of the generating set $X$ of $G$ is an involution and so any element $w$ of $G$ can be written in the form $ w = x_{1}x_{2}...x_{r}$ for $x_{j} \in X$. So the length, $l(w)$, of $w \in G$, will be the smallest $r$ such that $w = x_{1}x_{2}...x_{r}$. 

\begin{lemma}\label{Length} For all $w_{1}, w_{2} \in G$ and $x \in X$ the following properties hold.
\begin{itemize}
\item[(1)] $l(w_{1}) = l(w_{1}^{-1})$.
\item[(2)] $l(w_{1}) = 1 \text{ if and only if }  w_{1} \in X$.
\item[(3)] $l(w_{1}w_{2}) \leq l(w_{1}) + l(w_{2})$.
\item[(4)] $l(w_{1}w_{2}) \geq l(w_{1}) - l(w_{2})$.
\item[(5)] $l(w_{1}) - 1 \leq l(w_{1}x) \leq l(w_{1}) + 1$.
\item[(6)] $l(w_{1}x) = l(w_{1}) \pm 1$ and $l(xw_{1}) \neq l(w_{1})$
\end{itemize}
\end{lemma}
\begin{proof} The proof is analogous to the Coxeter case, see \cite[Section 5.2]{Humphreys}.
\end{proof} 
\begin{remark}\label{generalgroups2} The properties $(1)$, $(3)$ and $(4)$ in Lemma $\ref{Length}$ hold for a group with arbitrary presentation.

It is easy to see that $(2)$ and $(6)$ fail for the trivial group given by the presentation $\langle x \vert x = e \rangle$. However, the property
\begin{center}

\begin{itemize}
    \item[($2^{\ast})$] $w_{1} \in X\cup X^{-1} $ implies $ l(w_{1}) \leq 1$
\end{itemize} 

\end{center} holds for any group with group presentation $\langle X \vert R \rangle$. We can then apply $(2^{\ast})$, $(3)$ and $(4)$ to prove that $(5)$ also holds for a group with arbitrary presentation. 
\end{remark}
We are now ready to prove Proposition $\ref{decomposition}$. 

%We recall in the previous section we defined the following sets for a given subset $I \subseteq X$: $$G^{I} = \lbrace w \in G: l(wx_{i}) > l(w) \hspace{0.25cm} \forall \hspace{0.25cm} x_{i} \in I \rbrace$$
%$$^{I}G= \lbrace w \in G : l(x_{i}w) > l(w) \hspace{0.25cm} \forall \hspace{0.25cm} x_{i} \in I \rbrace$$
%\vspace{0.3cm}
%\noindent \textit{Proof of Proposition $\ref{decomposition}$.}
\begin{proof}[Proof of Proposition $\ref{decomposition}$] We proceed by induction on $l(w)$. If $l(w) = 1$ then $w = x$, for some $x \in X$. If $x \in I$, then we choose $a = e, b = x$ and, by Lemma $\ref{Length}\hspace{1mm}(2)$, this is the desired factorisation. If $x \notin I$, then we claim that $l(xy) > l(x)$ for all $y \in I$ and so we choose $a = x, b = e$. Taking any $y \in I$, if $l(xy) < l(x)$ then $l(xy) = 0$ as $l(x) = 1$, by Lemma $\ref{Length}\hspace{1mm}(2)$. Thus $xy = e$. It follows that $x = y \in I$, contradicting that $x \notin I$. As $l(xy) \neq l(x)$ by Lemma $\ref{Length}\hspace{1mm}(6)$, it must be that $l(xy) > l(x)$. 
 
Suppose $l(w) = r \geq 1$ and that the statement holds for every element of $G$ of shorter length. If $w \in G^{I}$ then we choose $b=e$ and $a =w$. Similarly, if $w \in G_{I}$ then we choose $a=e$ and $b=w$. So we need only consider the case when $w \notin G^{I}$ and $w \notin G_{I}$. 

As $w \notin G^{I}$, there exists $x \in I$ such that $l(wx) < l(w)$. By Lemma $\ref{Length}\hspace{1mm}(5)$, $l(wx) = l(w)-1 < r$. 
By induction, there exists $a' \in G^{I}$ and $b' \in G_{I}$ such that $wx = a'b'$ and $$l(wx) = l(w)-1 = l(a') + l(b').$$ Let $a=a'$ and $ b = b'x$. Then $ab = a'b'x = (wx)x = wx^{2} = w.$ It remains to show that $l(b'x)=l(b')+1$, giving $l(a) + l(b) = l(a') + l(b'x) = l(a')+l(b') +1 = l(wx) + 1 = l(w)$.

We assume, for a contradiction, that $l(b'x) < l(b')$. That is, by Lemma $\ref{Length}\hspace{1mm}(5)$, $l(b'x) = l(b')-1$. By the above, $wx = a'b'$, so $w = a'b'x$, and $l(wx) = l(a') + l(b')$. Thus \begin{align*}
l(w) = l(a'b'x) &\leq l(a')+l(b'x) \\
&= l(a') + (l(b')-1)\\
&=(l(a')+l(b'))-1 \\
&=l(wx)-1 \\
&<l(wx),
\end{align*} contradicting the fact that $l(wx) < l(w)$.  Since $l(b'x) \neq  l(b')$ by Lemma $\ref{Length}\hspace{1mm}(6)$, we have $l(b'x) > l(b')$ and so $l(b'x)=l(b')+1$ by Lemma $\ref{Length}\hspace{1mm}(5)$. Therefore, $l(w) = l(a)+l(b)$. Finally, we note that $a=a' \in G^{I}$ and, as $x, b' \in G_{I}$, we have that $b \in G_{I}$. Thus we have obtained the required factorisation of $w$. 

\end{proof}

\begin{remark} By applying Proposition $\ref{decomposition}$ to $w^{-1}$, it can be shown that, for any $I \subseteq X$, every element $w \in G$ has a factorisation $w = ab$ for some $ a \in G_{I}, b \in $ $^{I}G$ such that $l(w) = l(a) + l(b)$. 
\end{remark}

\begin{remark} Proposition $\ref{decomposition}$ does not hold in general for groups with an arbitrary group presentation. A counterexample is given by the the Klein four-group, $\mathcal{V} = \lbrace e, i, j, k \rbrace$ \cite[Section 44.5]{K4}, which has group presentation:

$$\mathcal{V} =  \langle i, j, k \vert i^{2} = j^{2} = k^{2} = ijk = e \rangle.$$

Taking $w = j$, this is a unique reduced expression of $w$. We have $\mathcal{V}_{\lbrace i \rbrace} = \lbrace e, i \rbrace$ and so $j \notin \mathcal{V}_{\lbrace i \rbrace}$. However, $ji = k$ and so $l(ji) = l(j)$, meaning $w \notin \mathcal{V}^{\lbrace i \rbrace}$. Thus $j$ has no reduced factorisation with respect to $\mathcal{V}_{\lbrace i \rbrace}$. 

\end{remark}

As stated in \cite[Proposition 2.4.4]{Bjorner}, in the Coxeter case this factorisation exists and is furthermore unique. The element in the factorisation lying in the set $W^{I}$ can be shown to be the unique element of $wW_{I}$ of minimal length \cite[Proposition 1.10]{Humphreys}. The uniqueness of these minimal length coset elements distinguish them as coset representatives and they are referred to as the \textit{minimal coset representatives} \cite[Section 1.10]{Humphreys}. Thus the set $wW_{I}\cap W^{I}$ contains only one element, namely the minimal coset representative of $wW_{I}$. The uniqueness of the factorisation for elements of Coxeter groups is a consequence of the Deletion Condition and is not a property that is transferable to the factorisations of elements in $G$ with respect to given parabolic subgroup, $G_{I}$. A counterexample proving this will be given in the next section.

However, in the cases when $I = X$ and $\vert I \vert = 1$ the factorisation of all elements of the group with respect to the corresponding parabolic subgroup will be unique.

\begin{lemma}\label{unique case} If $I = X$ or $I = \lbrace x \rbrace$ for some $x \in X$, then for all $w \in G$, $wG_{I}\cap G^{I}$ contains a unique element. Hence the factorisation of each element of $G$ with respect to $G_{I}$ is unique. 
\end{lemma}
\begin{proof} If $I = X$ then $G_{I} = G$ and $G^{I} = \lbrace e \rbrace$, thus $wG_{I}\cap G^{I} = \lbrace e \rbrace$. 

If $I = \lbrace x \rbrace$ for some $x \in X$, then $G_{I} = \lbrace e, x \rbrace$, as $X$ is a fixed set of involutions generating $G$, and $wG_{I} = \lbrace w, wx \rbrace$ for any $w \in G$. By Lemma $\ref{Length} (6)$, $l(wx) = l(w) \pm 1$ meaning that exactly one of $w$ or $wx$ is an element of $G^{I}$ and thus the unique element in $ wG_{I}\cap G^{I}$. 

We conclude that, in both cases, the factorisation of any $w \in G$ with respect to $G_{I}$ is unique as if $w = ab = a'b'$ where $a, a' \in G^{I}$ and $b, b' \in G_{I}$, then $a, a' \in wG_{I}\cap G^{I}$. Thus $a = a'$ and consequently $b = b'$. 
\end{proof}

\section{Non-uniqueness of factorisations.}\label{clustergroups}

In this section, we present a counterexample demonstrating that the factorisations, shown to exist by Proposition \ref{decomposition}, for elements of a group with a presentation whose generators are involutions and whose relations are of even length are not necessarily unique.

Specifically, we consider the finite group, $G$, arising from the group presentation $\langle t_{1}, t_{2}, t_{3} |  R \rangle$, where $R$ is the following set of relations:

\begin{itemize}
\item[$(a)$] $t_{1}^{2} = t_{2}^{2} = t_{3}^{2} = e$ \textit{(each generator is an involution)}.
\item[$(b)$] $t_{1}t_{2}t_{1} = t_{2}t_{1}t_{2}, \hspace{0.3cm}  t_{1}t_{3}t_{1} = t_{3}t_{1}t_{3},  \hspace{0.3cm}  t_{2}t_{3}t_{2} = t_{3}t_{2}t_{3}$ \textit{(the braid relations)}. 
\item[$(c)$]  $t_{1}t_{2}t_{3}t_{1} = t_{2}t_{3}t_{1}t_{2} = t_{3}t_{1}t_{2}t_{3}$ \textit{(the cycle relation)}.
\end{itemize} 

We note that the above group presentation of $G$ satisfies conditions $(\RomanNumeralCaps{1})$ and $(\RomanNumeralCaps{2})$, thus Proposition \ref{decomposition} holds for elements of $G$ with respect to a given parabolic subgroup. 

\begin{remark} This group presentation is a \textit{cluster group presentation}. Cluster groups are groups defined by presentations arising from cluster algebras. It was first shown in  \cite{BM} and then in \cite{GM} that a group presentation could be associated to a quiver appearing in a seed of a cluster algebra of finite type and that the corresponding group is invariant under mutation of the quiver \cite[Lemma 2.5]{GM}, \cite[Theorem 5.4]{BM}. For detailed definitions of cluster algebras and quiver mutation, see \cite[Chapter 2]{Marsh}. It is the group presentation based on the work done in \cite{GM}, that was considered more generally in \cite{Webster} where, due to the context, the corresponding groups were labelled \textit{cluster groups}. Each quiver appearing in a cluster algebra of finite type is mutation-equivalent to an oriented Dynkin diagram \cite[Theorem 1.4]{FZ} and the corresponding cluster group presentation is precisely a Coxeter presentation. Consequently, a cluster group associated to a mutation-Dynkin quiver is isomorphic to the finite reflection group of the same Dynkin type. The presentation defined above is the cluster group presentation associated to the quiver of mutation-Dynkin type $A_{3}$ in Figure \ref{fig: quiver}. 
\begin{figure}[!htb]
\center{
\begin{tikzpicture}
\node (v0) at (0:0)[text width=0.5cm] {$1$};
\node (v1) at (0:2)[text width=0.5cm] {$2$};
\node (v6) at (30:2) {};
\node (v4) at (0:4)[text width=0.5cm] {$3$};
%\node (v5) at (-15:-2)[text width=0.5cm] {$Q:$};

\draw[<-] (v0) -- (v1);
\draw[->] (v4) -- (v1);
\draw[<-] (v4) .. controls (v6) .. (v0);
\end{tikzpicture}
\caption{\label{fig: quiver} $G$ is the cluster group corresponding \\ to the above quiver of mutation-Dynkin type $A_{3}$.}}
\end{figure}
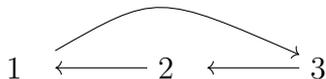
\end{remark}

It was shown in \cite[Lemma 3.10]{Webster} that an isomorphism between a cluster group associated to a mutation-Dynkin quiver of type $A_{n}$ and the symmetric group on $n+1$ elements, denoted by $\Sigma_{n+1}$, can be constructed explicitly from the quiver. It follows from \cite[Lemma 3.10]{Webster} that the following map defines an isomorphism between $G$ and $\Sigma_{4}$. 

\begin{lemma}\label{braidgraphiso1} \cite[Proposition 3.4]{Ser}, \cite[Lemma 3.10]{Webster} There exists an isomorphism $\pi: G \longrightarrow \Sigma_{4}$ given by 
\begin{align*}
\pi: &t_{1} \longmapsto (1, 2), \\
\pi: &t_{2} \longmapsto (2, 3), \\
\pi: &t_{3} \longmapsto (2, 4).
\end{align*}
\end{lemma}

From Lemma $\ref{braidgraphiso1}$, we conclude that $G$ contains 24 distinct elements. Moreover, we can use this isomorphism to determine that $t_{i}t_{j}t_{i}t_{k} = t_{k}t_{i}t_{j}t_{i}$ for all combinations of pairwise distinct $i, j, k$. Using this together with the group relations, we construct the Cayley graph of $G$ with respect to this presentation, which is displayed in Figure \ref{fig: Cayley graph}.

Our goal is to choose a subset, $I$, of the generating set of $G$ such that we can find an element of $G$ with two distinct factorisations with respect to the corresponding parabolic subgroup, thus proving the factorisations shown to exist in Proposition \ref{decomposition} are not necessarily unique. Due to Lemma \ref{unique case}, $I$ must contain two distinct elements. 

Take $I = \lbrace t_{1}, t_{2} \rbrace$. From the Cayley graph of $G$ (Figure \ref{fig: Cayley graph}) we have \begin{center} $G^{I}  = \lbrace e, \hspace{0.2cm} t_{3},\hspace{0.2cm} t_{1}t_{3},\hspace{0.2cm} t_{2}t_{3}, \hspace{0.2cm}t_{1}t_{2}t_{3},\hspace{0.2cm} t_{2}t_{1}t_{3} \rbrace. $
\end{center}

For $ w = t_{2}t_{3}t_{1}t_{2}$, we have $a = t_{2}t_{3} \in G^{I} $ and $ b = t_{1}t_{2} \in G_{I}$. Using the Cayley graph of $G$ in Figure \ref{fig: Cayley graph}, we observe that each of these are all reduced expressions and so $a$ and $b$ yield a factorisation of $w$ as given in Proposition \ref{decomposition}. Moreover, we can apply the cycle relation to $w$ to obtain another reduced expression of $w$, $w = t_{1}t_{2}t_{3}t_{1}$. Let $a' = t_{1}t_{2}t_{3} \in G^{I} $ and $b'=t_{1} \in G_{I}$. As this expression for $a'$ is a subexpression of a reduced expression of $w$, it must be reduced. From Lemma \ref{braidgraphiso1} we conclude that $t_{i} \neq e$ for $i = 1, 2, 3$, thus $t_{1}$ is also reduced. Clearly, as $a$, $a'$ and $b$, $b'$ are of different lengths, these are distinct pairs of elements in $G^{I}$ and $G_{I}$, respectively. Thus, we obtain distinct factorisations $w = ab = a'b'$ where $l(w) = l(a) + l(b) = l(a') + l(b')$ for $a, a' \in G^{I}$ and $b, b' \in G_{I}$. 

This counterexample demonstrates that for a group $G$ with presentation $\langle X \vert R \rangle $ satisfying conditions $(\RomanNumeralCaps{1})$ and $(\RomanNumeralCaps{2})$, unlike in the Coxeter case, it is possible for the set $wG_{I}\cap G^{I}$ to have more than one element for some $I \subseteq X$ and $w \in G$. In the counterexample above, two distinct elements in this set are given by $a = t_{2}t_{3}$ and $a'= t_{1}t_{2}t_{3}$ for $ w = t_{2}t_{3}t_{1}t_{2}$. 

Furthermore, for a Coxeter group, $W$, the unique factorisation of an element $w \in W$ with respect to a parabolic subgroup, $W_{I}$, is determined by the unique minimal length element of the coset $wW_{I}$ \cite[Corollary 2.4.5]{Bjorner}. In the more general case, the same example used to demonstrate that the factorisation is not necessarily unique can also be used to show that minimal length elements of $wG_{I}$ do not necessarily yield factorisations of $w$.  

Indeed, from the Cayley graph in Figure \ref{fig: Cayley graph}, we have 
\begin{center} $G_{I}  = \lbrace e, \hspace{0.2cm} t_{1},\hspace{0.2cm} t_{2},\hspace{0.2cm} t_{1}t_{2}, \hspace{0.2cm}t_{2}t_{1},\hspace{0.2cm} t_{1}t_{2}t_{1} = t_{2}t_{1}t_{2} \rbrace. $
\end{center} 
Choosing $w = t_{1}t_{2}t_{3} \in G^{I}$, we observe from the Cayley graph that this is a unique reduced expression of $w$ and so the only factorisation with respect to $G_{I}$ is obtained by taking $a = w$ and $b = e$. However, \begin{center} $wG_{I}  = \lbrace t_{1}t_{2}t_{3}, \hspace{0.2cm} t_{1}t_{2}t_{3}t_{1},\hspace{0.2cm} t_{1}t_{2}t_{3}t_{2},\hspace{0.2cm} t_{2}t_{3}t_{1}, \hspace{0.2cm}t_{2}t_{3}t_{2},\hspace{0.2cm} t_{2}t_{3} \rbrace.$\end{center} Thus no factorisation of $w$ can be obtained from the minimal length element, $t_{2}t_{3}$,  of $wG_{I}$. 

The property that this factorisation is unique for elements of $W$, and determined by the unique minimal length element of $wW_{I}$, is a result of the Deletion Condition \cite[Theorem 1.7]{Humphreys}. The Deletion Condition is a characterising result for Coxeter groups. That is, if $W$ is a group and $S$ a set of involutions generating $W$, then $(W, S)$ has the Deletion Condition if and only if $(W, S)$ is a Coxeter system \cite[Theorem 1.5.1]{Bjorner}. Considering the proof of \cite[Proposition 1.10(c)]{Humphreys}, we see that the factorisations of $w$ with respect to $G_{I}$, as taken in the above counterexample, are no longer unique or determined by the minimal length elements of $wG_{I}$ because, without the Deletion Condition, we can no longer omit certain factors from a non-reduced expression of $w$ and leave $w$ unchanged. 

\newpage

\begin{figure}[H]
\center{
\begin{tikzpicture}
   \newdimen\R
   \R=2cm
   
   %Hex1
  \draw[xshift=-1.49\R, yshift=-.88\R] (0:\R) \foreach \x in {60, 120, 180, 240, 300, 360} {  -- (\x:\R) };
\foreach \x in {120, 240} 
 \node[xshift=-1.49\R, yshift=-.88\R, inner sep=1pt,circle,draw,fill,label={}] at (\x:\R) {};
 \node[xshift=-1.49\R, yshift=-.88\R, inner sep=1.5pt,circle,draw,fill=orange,label={}] at (180:\R) {};
\node[xshift=-2.2\R, yshift=-.88\R, align=center] at (180:\R) {$t_{3}t_{2}t_{1}t_{2} = t_{3}t_{1}t_{2}t_{1}$ \\ \color{gray} $(1, 3)(2, 4)$};
\node[xshift=-2\R, yshift=-.88\R, align=center] at (120:\R) {$t_{3}t_{1}t_{2}$ \\ \color{gray} $(2, 3, 1, 4)$};
\node[xshift=-2\R, yshift=-.88\R, align=center] at (240:\R) {$t_{3}t_{2}t_{1} $ \\ \color{gray} $(1, 3, 4, 2)$};

  \draw[xshift=-1.49\R, yshift=-2.62\R] (0:\R) {(120:\R) -- (180:\R) };
 \node[xshift=-1.49\R, yshift=-2.62\R, inner sep=1.5pt,circle,draw,fill=pink,label={}] at (180:\R) {};
 \node[xshift=-1.49\R, yshift=-2.95\R, align=center] at (180:\R) {$t_{3}t_{2}t_{1}t_{3} $ \\ \color{gray}$(4, 1, 3)$};         
 
 \draw[xshift=-1.49\R, yshift=.86\R] (0:\R) {(360:\R) -- (60:\R) };
 \node[xshift=-1.49\R, yshift=.86\R, inner sep=1.5pt,circle,draw,fill=blue,label={}] at (60:\R) {};
\node[xshift=-1.49\R, yshift=1.18\R, align=center] at (60:\R) {$t_{3}t_{1}t_{3}t_{2} = t_{1}t_{3}t_{1}t_{2}$ \\ \color{gray} $(1, 4, 3)$};

   %Hex2
\draw (0:\R) \foreach \x in {60,120, 180, 240, 300, 360} {  -- (\x:\R) };
\foreach \x in {60,120, 180, 240, 300, 360} 
 \node[inner sep=1pt,circle,draw,fill,label={}] at (\x:\R) {};

\node[xshift=0.4\R, align=center] at (60:\R) {$t_{1}t_{3}$ \\ \color{gray}$(2, 4, 1)$};
\node[xshift=-0.35\R, yshift=0.1\R, align=center] at (120:\R) {$t_{3}t_{1}t_{3} = t_{1}t_{3}t_{1}$ \\ \color{gray}$(1, 4)$};
\node[xshift=0.4\R, align=center] at (180:\R) {$t_{3}t_{1}$ \\ \color{gray}$(1, 4, 2)$};
\node[xshift=-0.35\R, align=center] at (240:\R) {$t_{3}$ \\ \color{gray}$(2, 4)$};
\node[xshift=0.2\R, align=center] at (300:\R) {$e$ \\ \color{gray}$()$};
\node[xshift=-0.35\R, align=center] at (360:\R) {$t_{1}$ \\ \color{gray}$(1, 2)$};

\draw[xshift=-1.49\R, yshift=.86\R] (0:\R) {(180:\R) -- (240:\R) };
 \node[xshift=-1.49\R, yshift=.86\R, inner sep=1.5pt,circle,draw,fill=red,label={}] at (180:\R) {};
\node[xshift=-1.49\R, yshift=1.18\R, align=center] at (180:\R) {$t_{3}t_{1}t_{2}t_{3}$ \\ \color{gray} $(1, 4)(2, 3)$};

%Hex3
     
\draw[xshift=1.49\R, yshift=.86\R] (0:\R) \foreach \x in {60,120,...,360} {  -- (\x:\R) };
\foreach \x in {120, 360} 
 \node[xshift=1.49\R, yshift=.86\R, inner sep=1pt,circle,draw,fill,label={}] at (\x:\R) {};
 \node[xshift=1.49\R, yshift=.88\R, inner sep=1.5pt,circle,draw,fill=green,label={}] at (60:\R) {};

\node[xshift=2\R, yshift=1\R, align=center] at (60:\R) {$t_{1}t_{3}t_{2}t_{3} = t_{1}t_{2}t_{3}t_{2}$ \\ \color{gray}$(1, 2)(3, 4)$};
\node[xshift=1\R, yshift=1.01\R, align=center] at (120:\R) {$t_{1}t_{3}t_{2} $ \\ \color{gray}$(2, 3, 4, 1)$};
\node[xshift=1\R, yshift=.86\R, align=center] at (360:\R) {$t_{1}t_{2}t_{3}$ \\ \color{gray}$(2, 4, 3, 1)$};

 \draw[yshift=1.74\R] (0:\R) {(360:\R) -- (60:\R) };
\node[xshift=-0.025\R, yshift=1.78\R, inner sep=1.5pt,circle,draw,fill=pink] at (60:\R) {};
\node[yshift=2\R, align=center] at (60:\R) {$t_{1}t_{3}t_{2}t_{1}$ \\ \color{gray} $(4, 1, 3)$};

\draw[xshift=2.98\R] (0:\R) {(120:\R) -- (60:\R) };
 \node[xshift=2.98\R, inner sep=1.5pt,circle,draw,fill=pink,label={}] at (60:\R) {};
\node[xshift=3.35\R, yshift=0.5, align=center] at (60:\R) {$t_{1}t_{2}t_{3}t_{1}$ \\ \color{gray} $(1, 4, 3)$};

   %Hex4
\draw[yshift=-1.74\R] (0:\R) \foreach \x in {60,120, 180, 240, 300, 360} {  -- (\x:\R) };
\foreach \x in {60,120, 180, 240, 300, 360} 
\node[yshift=-1.74\R, inner sep=1pt,circle,draw,fill,label={}] at (\x:\R) {};

\node[xshift=0.4\R, yshift=-1.74\R, align=center] at (180:\R) {$t_{3}t_{2}$ \\ \color{gray}$(3, 4, 2)$};
\node[xshift=-0.31\R, yshift=-1.74\R, align=center] at (240:\R) {$t_{3}t_{2}t_{3} = t_{2}t_{3}t_{2}$ \\ \color{gray}$(3, 4)$};
\node[xshift=0.4\R, yshift=-1.74\R, align=center] at (300:\R) {$t_{2}t_{3}$ \\ \color{gray}$(2, 4, 3)$};

\draw[xshift=-1.49\R, yshift=-2.62\R] (0:\R) \foreach \x in {60,120,...,360} { (300:\R) -- (360:\R)};
\node[xshift=-1.49\R, yshift=-2.62\R, inner sep=1.5pt,circle,draw,fill=green] at (300:\R) {};
\node[xshift=-1.49\R, yshift=-2.88\R, align=center] at (300:\R) {$t_{3}t_{2}t_{3}t_{1} = t_{2}t_{3}t_{2}t_{1}$ \\ \color{gray} $(1, 2)(3, 4)$};

   %Hex5
\draw[xshift=1.49\R, yshift=-.88\R] (0:\R) \foreach \x in {300, 360, 60} {  -- (\x:\R) };
\foreach \x in {60,120, 300, 360} 
 \node[xshift=1.49\R, yshift=-.88\R, inner sep=1pt,circle,draw,fill,label={}] at (\x:\R) {};

\node[xshift=1.9\R, yshift=-.88\R, align=center] at (60:\R) {$t_{1}t_{2}$ \\ \color{gray}$(2, 3, 1)$};
\node[xshift=1.2\R, yshift=-.9\R, align=center] at (240:\R) {$t_{2}$ \\ \color{gray}$(2, 3)$};
\node[xshift=1.9\R, yshift=-.9\R, align=center] at (300:\R) {$t_{2}t_{1}$ \\ \color{gray}$(1, 3, 2)$};
\node[xshift=1.14\R, yshift=-.84\R, align=center] at (360:\R) {$t_{1}t_{2}t_{1} = t_{2}t_{1}t_{2}$ \\ \color{gray}$(1, 3)$};

\draw[xshift=2.98\R, yshift=-1.74\R] (0:\R) {(120:\R) -- (60:\R) };
 \node[xshift=2.98\R, yshift=-1.74\R, inner sep=1.5pt,circle,draw,fill=orange,label={}] at (60:\R) {};
\node[xshift=3.5\R, yshift=-1.7\R, align=center] at (60:\R) {$t_{1}t_{2}t_{1}t_{3} = t_{2}t_{1}t_{2}t_{3}$ \\ \color{gray} $(1, 3)(2, 4)$};

   %Hex6 

\draw[xshift=1.49\R, yshift=-2.62\R] (0:\R)  \foreach \x in {60,120, 180, 240, 300, 360} {  -- (\x:\R) };
\foreach \x in {240, 360} 
 \node[xshift=1.49\R, yshift=-2.62\R, inner sep=1pt,circle,draw,fill,label={}] at (\x:\R) {};
 \node[xshift=1.49\R, yshift=-2.62\R, inner sep=1.5pt,circle,draw,fill=blue,label={}] at (300:\R) {};

\node[xshift=1\R, yshift=-2.62\R, align=center] at (240:\R) {$t_{2}t_{3}t_{1}$ \\ \color{gray}$(1, 4, 3, 2)$};
\node[xshift=2.3\R, yshift=-2.62\R, align=center] at (300:\R) {$t_{2}t_{3}t_{1}t_{3} = t_{2}t_{1}t_{3}t_{1} $ \\ \color{gray}$(1, 4)(2, 3)$};
\node[xshift=1\R, yshift=-2.62\R, align=center] at (360:\R) {$t_{2}t_{1}t_{3}$ \\ \color{gray}$(2, 4, 1, 3)$};

\draw[xshift=2.98\R, yshift=-1.74\R] (0:\R) {(240:\R) -- (300:\R) };
\node[xshift=3.3\R, yshift=-1.74\R, align=center] at (300:\R) {$t_{2}t_{1}t_{3}t_{2} $ \\ \color{gray}$(4, 1, 3)$};
 \node[xshift=2.98\R, yshift=-1.74\R, inner sep=1.5pt,circle,draw,fill=pink,label={}] at (300:\R) {};

\draw[xshift=1.49\R, yshift=-4.36\R] (0:\R) {(120:\R) -- (180:\R) };
 \node[xshift=1.49\R, yshift=-4.36\R, inner sep=1.5pt,circle,draw,fill=red,label={}] at (180:\R) {};
 \node[xshift=1.49\R, yshift=-4.8\R, align=center] at (180:\R) {$t_{2}t_{3}t_{1}t_{2} $ \\ \color{gray}$(1, 4, 3)$};

\end{tikzpicture}
\caption{\label{fig: Cayley graph} The Cayley graph of $(G, \lbrace t_{1}, t_{2}, t_{3} \rbrace)$ and the corresponding permutations of each element under $\pi$. The coloured nodes denote an identification of vertices.}}
\end{figure}
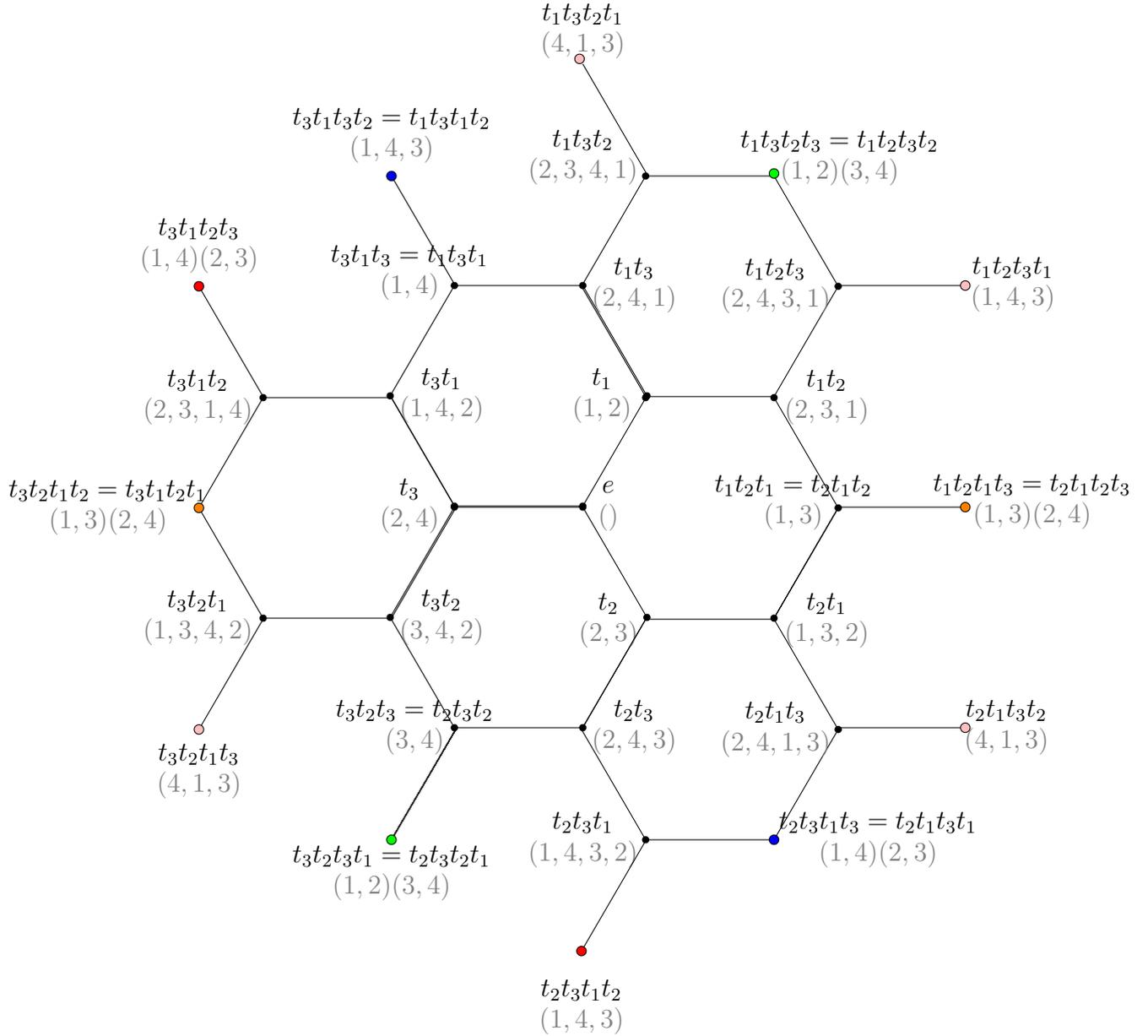

%\section{Declarations}

%\subsection{Funding} This research was completed as a part of the author's PhD thesis, which was funded by a University of Leeds 110 Anniversary Research Scholarship.

%\subsection{Conflicts of interest/Competing interests} Not applicable.

%\subsection{Availability of data and material} Not applicable.

%\subsection{Code availability} Not applicable.

%\begin{acknowledgements}
%\end{acknowledgements}




 \end{document}